\documentclass[10pt,a4paper]{amsart}
\usepackage{graphicx}

\usepackage[draft]{optional}
\usepackage[style=alphabetic]{biblatex}

\usepackage{amsfonts, amsmath, wasysym}
\usepackage{stmaryrd}
\usepackage{faktor}
\usepackage{bbm}

\usepackage{tikz}
\usetikzlibrary{shapes, positioning}
\usepackage{amssymb,amsthm,paralist}

\usepackage{latexsym}
\usepackage{url}

\definecolor{darkgreen}{rgb}{0,0.5,0}
\definecolor{darkred}{rgb}{0.7,0,0}
\usepackage[colorlinks, 
citecolor=darkgreen, linkcolor=darkred]{hyperref}

\theoremstyle{plain}
\numberwithin{equation}{section}

\newcommand{\al}{\alpha}
\newcommand{\be}{\beta}

\newcommand{\de}{\delta}

\newcommand{\la}{\lambda}

\renewcommand{\th}{\theta}

\newcommand{\ep}{\varepsilon}

\newcommand{\R}{\ensuremath{{\mathbb R}}}
\newcommand{\Sp}{\ensuremath{{\mathbb S}}}
\newcommand{\N}{\ensuremath{{\mathbb N}}}

\newcommand{\Z}{\ensuremath{{\mathbb Z}}}

\newcommand{\downto}{\downarrow}
\newcommand{\upto}{\uparrow}

\newcommand{\lap}{\Delta}

\DeclareMathOperator{\Vol}{Vol}

\newcommand{\beq}{\begin{equation}}
\newcommand{\eeq}{\end{equation}}
\newcommand{\beqa}{\begin{equation}\begin{aligned}}
\newcommand{\eeqa}{\end{aligned}\end{equation}}
\newcommand{\brmk}{\begin{rmk}}
\newcommand{\ermk}{\end{rmk}}
\newcommand{\partref}[1]{\hbox{(\csname @roman\endcsname{\ref{#1}})}}
\newcommand{\half}{\frac{1}{2}}

\newcommand{\Rm}{{\mathrm{Rm}}}
\newcommand{\Ric}{{\mathrm{Ric}}}

\usepackage{soul}

\newtheorem{thm}{Theorem}[section]

\newtheorem{lem}[thm]{Lemma}
\newtheorem{prop}[thm]{Proposition}
\newtheorem{defn}[thm]{Definition}
\newtheorem{rmk}[thm]{Remark}

\newtheorem{claim}[thm]{Claim}

\newcommand{\eps}{\varepsilon}

\usepackage{tikz}

\title{Steady Gradient Ricci Solitons with $O(p)\times O(q)$ Symmetry}
\date{}

\author{Lucas Lavoyer}
\address[Lucas Lavoyer]{Mathematisches Institut, Universit\"at M\"unster, 48149 M\"unster, Germany}
\email{lucas.lavoyer@uni-muenster.de}

\author{Luke T. Peachey}
\address[Luke T. Peachey]{Institut f\"ur Mathematik, Freie Universit\"at Berlin, 14195 Berlin, Deutschland}
\email{l.peachey@fu-berlin.de}

\begin{document}

\begin{abstract}
We find new examples of steady gradient Ricci solitons with positive curvature operator in dimensions four and above. Utilising a procedure first introduced by Lai, we construct examples with $O(p) \times O(q)$ symmetry in dimension $p+q$ for any pair of integers $p,q \geq 2$ and discuss their asymptotic geometry.
\end{abstract}

\maketitle

\section{Introduction}

Ricci solitons correspond to Ricci flows which are self-similar modulo scaling and reparameterisation. Amongst them, steady Ricci solitons correspond to those invariant under reparameterisation only and can be used to model the formation of Type II singularities under the Ricci flow \cite{Ham95}. A triple $(M,g,f)$ consisting of a complete Riemannian manifold $(M,g)$ and a smooth function $f:M \to \R$ is called a steady gradient Ricci soliton if
\begin{equation}\label{eq SGRS}
    \Ric(g) = \nabla^2 f.
\end{equation}
In this case, the steady soliton structure naturally induces an eternal Ricci flow on $M,$ $g(t) = \phi_t^*g$, where $\phi_t$ is the flow of the vector field $-\nabla f$ and $t\in (-\infty, \infty).$

Examples of Ricci solitons can be found by imposing symmetry ansatz. The strongest symmetry assumption one can make is rotational symmetry, which reduces \eqref{eq SGRS} to an ODE. Up to scaling, the only non-flat rotationally symmetric steady gradient Ricci solitons are Hamilton's cigar soliton in dimension two \cite{RFonSurfaces}, and Bryant's example in dimensions three and higher \cite{Bryant}. One may instead impose weaker symmetry assumptions. Conjectured to exist by Hamilton, Lai constructed the first examples of $O(1) \times O(n-1)$ symmetric steady gradient Ricci solitons with positive curvature operator in dimensions three and higher \cite{Lai24}. These solitons are referred to as flying wings due  to their shape at infinity, as the link of their asymptotic cone is an interval of length $\al \in (0,\pi).$ In \cite{Lai-flyingWingsForAnyCone}, Lai proves that there exists a 3-dimensional flying wing for any such \(\al \in (0,\pi).\) It is a consequence of their asymptotic geometry that these flying wings are collapsed in dimension three and non-collapsed in dimensions four and above. See also \cite{kahlerflyingwings} for work on Kähler flying wings, and \cite{CLL} for even more recent work on higher dimensional flying wings, which appeared after the first version of this work was made publicly available. 

More generally, one may consider the symmetry ansatz of the group $O(p) \times O(q)$ for any pair of integers $p,q \geq 2$ in dimension $p+q \geq 4$. If $(M^{p+q},g,f)$ denotes a complete steady gradient Ricci soliton with $O(p) \times O(q)$ symmetry, the isotropy representation at the unique fixed point of the action decomposes the tangent space into two irreducible subspaces. We denote by $\la_p$ and $\la_q$ the sectional curvatures of any plane contained entirely within one of these subspaces (see Section~\ref{prelimSection} for more details).

\begin{thm}\label{thmpq}
For any $p,q 
\in \Z_{\geq 2}$ and $\th \in (0,1)$, there exists a steady gradient Ricci soliton $\mathcal{S}^{p,q}_\th$ of dimension $p+q$ with positive curvature operator and $O(p) \times O(q)$ symmetry, such that the invariant sectional curvatures at the vertex satisfy $\la_p  = \th \cdot \la_q$.
\end{thm}

As with flying wings, we wish to understand the geometry at infinity for these solitons. We know that the ideal boundary of our Ricci soliton also admits a faithful isometric $O(p) \times O(q)$ action, from which we can deduce that it has the structure of a (not necessarily smooth) doubly warped product (see Theorem \ref{ideal bdry thm}). This structure leaves only a few possibilities for its topology. It would be interesting to check if the orbits of the $O(p) \times \{1_q \}$ subgroup in the ideal boundary of the steady solitons constructed above are non-collapsed, in which case the ideal boundary would be homeomorphic to a closed disk of dimension $p$ (see Remark~\ref{rmk pdisk}).

In fact, assuming the above to be true, we obtain some extra information on the asymptotic geometry. It is known from \cite{DimRedSteadySoltn} that non-collapsed steady gradient Ricci solitons with $\operatorname{sec}\geq 0$ always dimension reduce at infinity. In our case, following points going to infinity along the totally geodesic submanifold fixed by the action of the $\{1_p\} \times O(q)$ subgroup, we show that we can split a $p$-dimensional Euclidean factor under the same rescaling. From now on, for any $p,q \in \Z_{\geq 2}$ and $\th \in (0,1)$, let $\mathcal{S}^{p,q}_\th = (M,g,f,o)$ denote the steady Ricci soliton constructed in Theorem~\ref{thmpq}, and $\mathcal{S}^{p,q}_{\th}(\infty)$ its ideal boundary (see Section \ref{prelimSection} for the precise definition).

\begin{thm}\label{thmsplitting}
Under the additional assumption that the ideal boundary $\mathcal{S}^{p,q}_{\th}(\infty)$ is homeomorphic to a closed disk of dimension $p$ as in Remark~\ref{rmk pdisk}, then for any sequence of points $x_m$ which are fixed by the action of the $\{1_p\} \times O(q)$ subgroup satisfying
\[
d_g(x_m,o)\to \infty,
\]
the rescaled sequence $(M, \mathcal{R}_g(x_m)\cdot g(\mathcal{R}_g(x_m)^{-1}t),x_m)$ smoothly converges along a subsequence to the product of flat Euclidean space $\R^{p}$ with a rotationally symmetric ancient Ricci flow with positive curvature operator $(N^{q},g_N(t),x_\infty)$.
\end{thm}
\begin{rmk}
As a corollary to the theorem above, the splitting of $p$ flat directions allows one to argue exactly as in the proof of \cite[Corollary 3.5]{Lai24} to deduce that the solitons are non-collapsed if $q\geq 3$.
\end{rmk}

\textbf{Outline.} In Section~\ref{prelimSection} we introduce terminology and state useful results on gradient Ricci solitons, as well as from the theory of Alexandrov spaces and manifolds with symmetry. Sections~\ref{linkSection}~and~\ref{s4} contain the proof of Theorem~\ref{thmpq} which we briefly explain here in more detail. In \cite{Lai24}, Lai observed that given a family of expanding solitons with positive curvature operator asymptotic to cones with smooth links, as in \cite{Deruelle15} (see also Theorem \ref{Deruelle thm}), if the asymptotic volume ratio ($\operatorname{AVR}$) is going to zero along the sequence, by performing a blow-up analysis on the points of highest curvature, one extracts a steady Ricci soliton. The $\operatorname{AVR}$ going to zero is obtained by finding a suitable family of collapsing metrics on the sphere; it is not hard to see that if the volume of the links of the corresponding asymptotic cones goes to zero along the sequence, so does the asymptotic volume ratio of the expanding solitons. One of the main steps is therefore finding a family of asymptotic cones with asymptotic volume ratios going to zero which can be smoothed out by Deruelle's expanders. In Section~\ref{linkSection}, we carefully construct metrics on the sphere that satisfy these requirements and, additionally, admit faithful isometric $O(p)\times O(q)$ actions. We complete the proof of Theorem~\ref{thmpq} in Section~\ref{s4}. Finally, in Section~\ref{s5} we study the geometry at infinity of the steady solitons we have constructed and give a proof of Theorem \ref{thmsplitting}.\\

{\it Acknowledgements:} The authors would like to thank Man-Chun Lee for many insightful discussions. The first-named author was supported by the German Research Foundation (DFG) under Germany’s Excellence Strategy EXC 2044-390685587 ‘Mathematics Münster: Dynamics-Geometry-Structure’ and by the CRC 1442 ‘Geometry: Deformations and Rigidity’ of the DFG.

\section{Preliminaries}\label{prelimSection}

We collect here some well known results that will be useful later. We begin with a brief discussion on gradient Ricci solitons, with emphasis on the expanding and steady cases.

\begin{defn}
A triple $(M,g,f)$ consisting of a complete Riemannian manifold \((M,g)\) and a smooth function $f : M \to \R$, referred to as the soliton's potential function, is called a gradient Ricci soliton if it satisfies the equation
\begin{equation}\label{eq GRS}
\Ric(g)+ \frac{\ep}{2} \cdot g= \nabla^2 f,    
\end{equation}
for some $\eps \in \{-1,0,1\}$. If $\eps=1$ ($\eps=0$), we say that $(M,g)$ is an expanding (steady) gradient Ricci soliton.
\end{defn}

From now on, we assume $(M,g,f)$ is an expanding or steady gradient Ricci soliton with positive Ricci curvature. Equation \eqref{eq GRS} then implies that $f$ is a strictly convex function and hence has a unique critical point $o$ which we refer to as the \textit{vertex} of the soliton. We will often use the quadruple $(M,g,f,o)$ when representing the soliton to emphasize the vertex. By the addition of a constant we can always normalise the soiton to ensure that $f \geq 0$ and $f(o) = 0$. The following soliton identities can be found in \cite{CLN06, ChenDeruelle}. 

\begin{lem}
Let $(M^n,g,f,o)$ be a steady or expanding gradient Ricci soliton with positive Ricci curvature. If we assume the normalisation $f(o) = 0$, then the following identities hold:
\begin{align}
&\mathcal{R}_g + \frac{n}{2} = \lap f,\label{eq S1}\\
&\mathcal{R}_g + |\nabla f|^2  - f = \sup_{M} \mathcal{R}_g = \mathcal{R}_g(o),\label{eq S2}\\
&\nabla \mathcal{R}_g + 2\Ric(\nabla f) =0,\label{eq S3}\\
&\frac{1}{4}d_g^2(o,\cdot) \leq f \leq \frac{1}{4} \left(d_g(o,\cdot) + 2 \sqrt{\mathcal{R}_g(o)} \right)^2\label{eq S4}
\end{align}
where $\mathcal{R}_g$ denotes the scalar curvature.
\end{lem}

The following lemma relates the scalar curvature at the vertex with the asymptotic volume ratio. See \cite[Corollary 1.2]{sobolevIneq}, \cite[Theorem 4.3]{volGrowth}.

\begin{lem}\label{scalAVR_ratio}
For an expanding gradient Ricci soliton $(M,g,f,o)$ with positive Ricci curvature, we have the inequality
\begin{equation}\label{eqsnn}
\frac{\textnormal{vol}(B(o, \sqrt{\mathcal{R}_g(o)})}{(\sqrt{\mathcal{R}_g(o)})^n} \leq 4^n\cdot  \textnormal{AVR}(g).
\end{equation}
\end{lem}
\begin{proof}
Normalising so that $f(o) = 0$ and using \eqref{eq S4}, for every $s\geq \mathcal{R}(o)$ we have the inclusions
\begin{align}\label{eq; sub-level set incl}
  B\left(o, 2\sqrt{s} - 2\mathcal{R}(o)^{\half}\right) \subseteq     \{f \leq s\}  \subseteq B\left(o, 2\sqrt{s}\right),
\end{align} 
and hence
$$\textnormal{AVR}(g) = \lim_{s \upto \infty} \frac{\textnormal{vol}(\{f \leq s\})}{2^n \cdot s^{\frac{n}{2}}}.$$
Using the soliton equations and the co-area formula we have that
\begin{align*}
\frac{n}{2} \cdot \operatorname{vol}(\{f\leq s\}) &\leq \int_{\{f\leq s\}} \Delta f d\mu = \int_{\{f=s\}}|\nabla f|d\sigma\\
&\leq (s + \mathcal{R}_g(o)) \int_{\{f=s\}} |\nabla f|^{-1} d\sigma\\
&= (s + \mathcal{R}_g(o)) \frac{d}{ds}\operatorname{vol}(\{f\leq s\}),
\end{align*}
which when integrated between $0<r<s<\infty$ yields
\begin{align*}
\frac{\operatorname{vol}(\{f\leq s\})}{\operatorname{vol}(\{f\leq r\})} \geq \exp\left(\frac{n}{2} \cdot \int_r^s (x + \mathcal{R}_g(o))^{-1} dx \right)= \left(\frac{\mathcal{R}_g(o) + s}{\mathcal{R}_g(o) + r}\right)^{\frac{n}{2}}.
\end{align*}
Combining this with the formula for the asymptotic volume ratio and taking $s \upto \infty$ gives
\begin{equation}\label{eqavr}
\frac{\textnormal{vol}(\{ f \leq r\})}{ \left( r + \mathcal{R}_g(o)\right)^{\frac{n}{2}}} \leq 2^n \cdot \textnormal{AVR}(g).
\end{equation}
Finally, setting $r = \frac{9}{4} \cdot \mathcal{R}_g(o)$ in \eqref{eqavr} and using the inclusion \eqref{eq; sub-level set incl}, we have \eqref{eqsnn}.
\end{proof}

The following result of Deruelle \cite{Deruelle15} shows the existence of Ricci expanders resolving non-negatively curved metric cones. Moreover, these Ricci expanders are unique within a suitable class of smoothly asymptotically conical Ricci expanders (see \cite[Definition 1.2]{Deruelle15} for a precise definition of an expander being smoothly asymptotic to a metric cone).

\begin{thm}\label{Deruelle thm}(Theorem 1.3, \cite{Deruelle15})
    Let $h$ be a smooth Riemannian metric on $\mathbb{S}^{n-1}$ with $\operatorname{Rm}(h)\geq 1.$ Then there exists a unique normalised expanding gradient Ricci soliton with non-negative curvature operator smoothly asymptotic to $(C(\mathbb{S}^{n-1}), dr^2+r^2h, \tfrac{r}{2}\partial_r).$
\end{thm}

\subsection{Metric geometry}

Here we recall some basic facts from Alexandrov geometry. For more details, we refer the reader to \cite{BBI}. For a Riemannian manifold $(M,g)$ with non-negative sectional curvature, let $xyz$ be a geodesic triangle in $M$ and consider the corresponding comparison triangle $\overline{xyz}$ in the Euclidean plane formed by sides of length $d(x,y),d(y,z)$ and $d(x,z)$. Denote by $\widetilde{\measuredangle} xyz$ the angle at $\overline{y}$ in this comparison triangle, for which $\widetilde{\measuredangle} xyz \leq \measuredangle xyz$. Fix $o \in M$ and denote by $\Gamma_o$ the collection of unit speed geodesic rays $\gamma:[0,\infty) \to M$ emanating from $o=\gamma(0)$. Given a pair of geodesic rays $\gamma, \sigma \in \Gamma_0$ and $s>0$, the comparison angle $s \mapsto  \widetilde{\measuredangle} \gamma(s) o \sigma(s)$ is decreasing, and in particular the limit
\[
\lim_{s \to \infty} \widetilde{\measuredangle} \gamma(s) o \sigma(s) \in [0,\pi]
\]
exists and is the angle at infinity between $\gamma$ and $\sigma$. The ideal boundary $M(\infty)$ is then the space of equivalence classes of rays in $\Gamma_o$ equipped with this metric, and forms a compact length space. The asymptotic cone is a metric cone over the ideal boundary, and it is isometric to the pointed Gromov--Hausdorff limit of any blow-down sequence of the manifold. The asymptotic cone being non-negatively curved can then be characterised by a lower bound on the curvature of the ideal boundary in the Alexandrov sense.

\begin{prop}\label{prop; cbb1}(\cite[Prop. 4.2.3]{Alex})
    Let $X$ be a complete metric space. Then $X$ is an Alexandrov space of curvature $\geq 1$ if and only if the metric cone $C(X)$ is not a line and is an Alexandrov space of curvature $\geq 0.$
\end{prop}

For each geodesic ray $\gamma \in \Gamma_o$, we may also consider its associated Busemann function
$$ B_\gamma(x) := \lim_{t \upto \infty} t - d(x,\gamma(t)), \quad \textnormal{ for each } x \in M.$$
Taking the supremum over all such rays $B := \sup_{\gamma \in \Gamma_0} B_\gamma$ gives the Busemann function about $o$, which is a well-defined convex function \cite[Proposition 1.3]{CG72} whose level sets are convex hypersurfaces diffeomorphic to spheres \cite[Proposition 2.4.3]{LK95}. The following result of Guijarro and Kapovitch states that rescaled level sets of the Busemann function equipped with their intrinsic  metric converge to the ideal boundary.

\begin{lem}[{\cite[Lemma 3.6]{LK95}}]\label{lem Busemann}
For each $s>0$, let $i_s$ denote the intrinsic metric on $\{B=s\}$ viewed as a subset of $M$. Then the rescaled level sets
$$ \left(\{B = s\}, s^{-1} \cdot i_s\right)  \rightarrow M(\infty),$$ 
in the Gromov-Hausdorff sense as $s \upto \infty$.
\end{lem}

\subsection{Symmetric Manifolds}
For any $p,q \in \Z_{\geq 2}$, let $n:=p+q$ and consider the natural action of $O(p) \times O(q)$  on $  \R^p \times \R^q = \R^n.$ We refer to this action as the standard action of $O(p) \times O(q)$ on $\R^n$.
\begin{defn}
A Riemannian manifold $(M^n,g)$ is $O(p) \times O(q)$ symmetric if it admits an isometric $O(p) \times O(q)$ action and there exists a diffeomorphism $M^n \to \R^n$ which is equivariant with respect to its standard action on $\R^n$. 
\end{defn}

If a Riemannian manifold is $O(p) \times O(q)$ symmetric, we denote by $\Sigma_p \cong \R^p$ and $\Sigma_q \cong \R^q$ the totally geodesic submanifolds given by the fixed points of the subgroups $\{1_p\} \times O(q)$ and $O(p) \times \{1_q\}$ respectively. Away from these points, our group action is free, and hence we have a Riemannian submersion into the orbit space. If $(M,g,f,o)$ is an expanding or steady gradient Ricci soliton with $\operatorname{Rm}(g)>0,$ where $o$ is the unique critical point of the potential function, and additionally, if the soliton is $O(p)\times O(q)$ symmetric, then the vertex $o$ is the unique fixed point under the action. We can therefore decompose the tangent space $T_oM = T_o\Sigma_p\oplus T_o\Sigma_q.$

Following the discussion above, we may consider the sectional curvature of a $2-$plane $\pi_p \subset T_o\Sigma_p$ (respectively $\pi_q \subset T_o\Sigma_q$). Theorem \ref{thmpq} then says that for any $\pi_p \subset T_o\Sigma_p,$ $\pi_q \subset T_o\Sigma_q,$ and $\theta \in (0,1)$, we have $$\operatorname{sec}(\pi_p)= \lambda_p = \theta \cdot \lambda_q=\theta \cdot \operatorname{sec}(\pi_q).$$

Finally, given a function $f:M\to \R$ and a compact Lie group $G$ acting on $M,$ we can define the $G-$average of $f$ as follows.

\begin{defn}\label{defnGav}
Let $G$ be a compact Lie group acting smoothly by isometries on a complete Riemannian manifold $(M,g)$. For any function $f : M \to \R$, we define its $G-$average by
$$  f_G(x) := \int_G f(\rho \cdot x) \ d\mu_G(\rho),$$
where $d\mu_G$ denotes the unit Haar measure on $G$.
\end{defn}

\begin{lem}\label{lem Gav}
Let $G$ be a compact Lie group acting smoothly by isometries on a complete Riemannian manifold $(M,g)$. If a function $f:M \to \R$ is $\la-$concave, then its $G-$average is also $\la-$concave.
\end{lem}
\begin{proof}
For any unit speed geodesic $\gamma:I \to M$ and $\rho \in G$, its image $\rho \cdot \gamma$ remains a unit speed geodesic, and hence
\begin{align*}
(f_G \circ \gamma)''(t) &= \int_{G} (f \circ(\rho \cdot \gamma))''(t)  \ d\mu_G(\rho)\\
&\leq \int_G \la \ d\mu_G(\rho) = \la,
\end{align*}
where in the second inequality we applied the definition of $f$ being $\la-$concave.
\end{proof}

\section{Smooth Collapsing Links}\label{linkSection}
The purpose of this section is to construct a continuous family of smooth $O(p) \times O(q)$ symmetric metrics on the sphere which have uniformly positive curvature operator, but which collapse to a upper hemisphere of lower dimension. These will be used later as the links of metric cones which we will resolve by gradient Ricci expanders. We note here that, unlike the results in \cite{CLL}, where the authors use delicate Ricci flow stability estimates to construct collapsing metrics on the sphere, our construction below is done explicitly and in a straightforward manner.

\begin{thm}\label{thmlink}
For any $p,q \in \Z_{\geq 2}$, there exists a continuous family of smooth metrics on the sphere $(\Sp^{p+q-1},h(t))_{t \in (0,1]}$ such that
\begin{enumerate}[(i)]
\item Each metric admits a faithful isometric $O(p)\times O(q)$ action.
\item\label{2.2curvature} The curvature operator $\Rm(h(t)) \geq 1$ for any $t \in (0,1]$;
\item\label{2.2collapsing} $(\Sp^{p+q-1},h(t)) \rightarrow \Sp_+^{p}$ in the Gromov--Hausdorff sense as $t \downto 0$.
\end{enumerate}
\end{thm}

To understand this theorem, consider the round sphere as a doubly warped product
$$ \Sp^{p+q-1} \cong \left[0,\frac{\pi}{2}\right] \times_{\cos(r)} {\Sp^{p-1}} \times_{\sin(r)} \Sp^{q-1}.$$

For each $t \in (0,1]$ modify the warping function on the $\Sp^{q-1}$ fibre by a factor of $t$, so that as $t \downto 0$, these singular metrics collapse to $\left[0,\frac{\pi}{2}\right] \times_{\cos(r)} {\Sp^{p-1}} \cong \Sp^p_+$. The desired family of metrics are given by suitable smoothing and rescaling of these singular metrics. Thus in order to prove Theorem~\ref{thmlink} we need a smooth warping function interpolating between $\sin(r)$ and $t \cdot \sin(r)$ which preserves the lower curvature bound.

To begin, we construct a suitable family of smooth approximations to the Heaviside function.

\begin{lem}\label{lemcutoff}
There exists a continuous family of smooth functions $\{\chi_{\de} : \R \to [0,1]\}_{\delta \in (0,1]}$ with the following properties
\begin{enumerate}[(i)]
 \item For each $\delta \in (0,1]$, $\chi_{\de}(\cdot)$ is a decreasing function.\label{HA1}
\item For each $\delta \in (0,1]$, there exists $\eta_{\de} > 0$ such that
    \begin{equation*}
        \chi_{\de}(x) = \begin{cases}
            1 &: x < \eta_{\de}\\
            0 &: x > \frac{\pi}{2} - \eta_{\de}
        \end{cases}.
    \end{equation*} \label{HA2}
\item We have pointwise convergence $\chi_{\de} \rightarrow \mathbbm{1}_{\{\leq 0\}}$ to the Heaviside function as $\delta \downto 0$. \label{HA3}
\item We have the following integral decay of the weighted derivative
\begin{equation*}
    \lim_{\de \downto 0} \sup_{r \in (0,\frac{\pi}{2})} \frac{\int_0^r -\chi_{\de}'(x) \sin(x) dx}{\sin(r)} = 0.
\end{equation*} \label{HA4}
\end{enumerate}
\end{lem}

\begin{proof}
Begin by defining smooth increasing functions $\beta,\gamma: \R \to [0,\infty)$ such that
\begin{itemize}
    \item $\beta(x) = \gamma(x) = 0$ for $x \leq 0$;
    \item $\beta(x) \geq x$ and $\gamma(x) = 1$ for $x 
    \geq \half$;
    \item $\beta'(x),\beta''(x) \geq 0$ for all $x \in \R$.
\end{itemize}
For each $\de \in (0,1]$, we then renormalised these functions as
$$\beta_{\de}(x):= \beta\left( \frac{x^{\de}}{\de} - 1\right), \quad \gamma_\de(x) := \gamma\left(1 + \frac{x}{\de} -\frac{\pi}{2\de}\right),$$
and define
$$ \chi_{\de}(x) := (1 - \gamma_{\de}(x)) \exp(-\be_{\de}(x)), \quad \forall x \in \R.$$
Since $\be$ and $\gamma$ are increasing, \eqref{HA1} follows immediately. Moreover
\begin{itemize}
    \item $\gamma_{\de}(x) = 1$ for $x \geq \frac{\pi}{2} - \frac{\de}{2}$;
    \item $\gamma_{\de}(x) = 0$ for $x \leq \frac{\pi}{2} - \de$;
    \item $\be_{\de}(x) = 0$ for $x \leq \de^{\frac{1}{\de}} $.
\end{itemize}
Thus \eqref{HA2} follows with $\eta_{\de} := \min\{ \de^{\frac{1}{\de}} , \frac{\de}{2}\}$. Also, \eqref{HA3} is an immediate consequence of the fact that $\frac{x^\de}{\de} \upto \infty$ as $\de \downto 0$ for any $x>0$. It remains to show \eqref{HA4}. We first note that
$$ 0 \leq -\chi_{\de}'(x) \leq (\gamma_{\de}'(x) + \be_{\de}'(x)) \exp(-\be_{\de}(x)),$$
and
\begin{align*}
 \limsup_{\de \downto 0} \frac{\int_0^r \gamma_{\de}'(x) \exp(-\be_{\de}(x)) \sin(x) dx }{\sin (r)} &\leq  \limsup_{\de \downto 0} \int_0^r \gamma_{\de}'(x) \exp(-\be_{\de}(x)) dx\\
 &\leq \limsup_{\de \downto 0} \int_{\frac{\pi}{2} - \de}^{\frac{\pi}{2}} \gamma_{\de}'(x) \exp(-\be_{\de}(x)) dx\\
 &\leq \limsup_{\de \downto 0} \exp(-\be_{\de}(\frac{\pi}{2} - \de)) = 0.   
\end{align*} 

Therefore, we may assume for a contradiction that there are sequences $\de_j \downto 0$ and $r_j \in (0,\frac{\pi}{2})$ satisfying the uniform lower bound
\begin{equation}\label{eqHA4c}
\frac{\int_0^{r_j} \be_{\de_j}'(x) \exp(-\be_{\de_j}(x)) \sin(x) dx}{\sin(r_j)} \geq c > 0.
\end{equation}
Using the definition of $\beta_\de$, it follows that 
$$ R_j := \frac{r_j^{\de_j}}{\de_j} - 1 > 0,$$
and so after potentially passing to a subsequence, we may assume that $R_j \rightarrow R \in [0,\infty]$. Utilising the change of variable $y=\frac{x^{\de_j}}{\de_j} - 1$, we find that
\begin{align*}
\int_0^{r_j} \be_{\de_j}'(x) \exp(-\be_{\de_j}(x)) \sin(x) dx &= \int_{\de_j^{\frac{1}{\de_j}}}^{r_j} \be_{\de}'(x) \exp(-\be_{\de}(x)) \sin(x) dx\\
&= \int_{0}^{R_j}  \beta'(y) \exp(-\beta(y)) \sin \left( \left( \de_j (y+1)\right)^{1/\de_j} \right) dy,
\end{align*}
and hence \eqref{eqHA4c} becomes
\begin{equation}\label{eqHA4cc}
\int_{0}^{\infty}  \beta'(y) \exp(-\beta(y)) f_j(y) dy \geq c,
\end{equation}
where $f_j:(0,\infty) \to [0,1]$ is the measurable function
$$f_j(y) := \frac{\sin \left( \left( \de_j (y+1)\right)^{1/\de_j} \right)}{\sin(r_j)} \cdot \mathbbm{1}_{(0,R_j)}(y).$$
It follows that $f_j(y) \downto 0$ as $j \rightarrow \infty$ for almost every $y > 0$. Indeed, in the case $R = 0$ there is nothing to show, so we may assume $R \in (0,\infty]$. If $y > R$, then $f_j(y)$ is eventually zero by virtue of the indicator function. If $0 < y < R$, for $j$ sufficiently large, $r_j > \left( \de_j (y+\eta+1)\right)^{1/\de_j}$ for some $\eta > 0$. Which when combined with $x/2 \leq \sin(x) \leq x$ for small positive $x$ yields
\begin{align*}
\limsup_{j \upto \infty} f_j(y) &\leq \lim_{j \upto \infty} \frac{\sin\left( \left( \de_j (y+1)\right)^{1/\de_j}\right)}{\sin \left( \left( \de_j (y+\eta+1)\right)^{1/\de_j}\right)}\\
&\leq \lim_{j \upto \infty} 2 \left( \frac{y+1}{y+\eta+1}\right)^{1/\de_j} = 0.
\end{align*}
Finally, using the fact that 
$$\int_0^\infty \beta'(y) \exp(-\beta(y)) dy = 1,$$ 
the dominated convergence theorem implies that the ratio on the left hand side of $\eqref{eqHA4cc}$ converges to zero. This is a contradiction, and therefore \eqref{HA4} is true.
\end{proof}

We may now use the smooth approximations of the Heaviside function from the previous lemma to construct an appropriate family of warping functions.

\begin{lem}\label{lemsmooth}
There exists a continuous family of smooth functions $\{\varphi_{t} : [0,\frac{\pi}{2}] \to [0,1]\}_{t \in (0,1]}$ with the following properties
\begin{enumerate}[(1)]
\item\label{Wsmooth} 
$$\varphi_{t}^{(even)}(0) = 0, \ \varphi_{t}'(0) = 1, \ \varphi_{t}^{(odd)}(\pi/2) = 0, \ \varphi_{t}(\pi/2) > 0.$$
\item\label{Wcollapse} 
$$\varphi_t(r) \leq 3t \cdot \sin(r), \quad \forall r \in (t,\frac{\pi}{2}].$$
\item\label{Wcurvature} $$\frac{-\varphi_t''}{\varphi_t} \geq \tan \cdot \frac{\varphi_t'}{\varphi_t} \geq \max\{ t, 1-t\}.$$
\end{enumerate}
\end{lem}

\begin{proof}
Consider the family of smooth approximations of the Heaviside function $\{\chi_{\de}\}_{\de \in (0,1]}$ constructed in Lemma~\ref{lemcutoff}. It follows immediately from properties \eqref{HA3} and \eqref{HA4} of these functions that there exists a continuous decreasing function $\de : (0,1] \to (0,1]$ such that
\begin{enumerate}[(a)]
 \item $\chi_{\de(t)}(t) \leq t$ for every $t \in (0,1]$;\label{eqdelta1}
\item $\int_0^r -\chi_{\de(t)}'(x) \sin(x) dx \leq \min\{ \frac{t^2}{1-t} , t \} \cdot \sin(r)$, for every $t \in (0,1]$ and $r \in [0,\frac{\pi}{2}]$.  \label{eqdelta2} 
\end{enumerate}
From now on we use the shorthand $\chi_t := \chi_{\de(t)}$. Consider the interpolation
$$ \varphi_{t}(r) := \int_0^r t \cos(x) + (1-t) \cdot  \chi_{t}(x) \cos(x) \ dx.$$
It is straightforward that $t \sin(r) \leq \varphi_{t}(r) \leq \sin(r)$. By \eqref{HA2} of Lemma~\ref{lemcutoff}, $\varphi_{t}'(r) = \cos(r)$ for $r$ in a neighbourhood of $0$, and $\varphi_{t}'(r) = t \cdot \cos(r)$ for $r$ in a neighbourhood of $\frac{\pi}{2}$. This gives us the boundary data \eqref{Wsmooth}. Next, using integration by parts and that $\chi_{t}$ is decreasing
\begin{align*}
 \varphi_{t}(r) &= \tan(r) \cdot \varphi_{t}'(r) - (1-t) \int_0^r \chi_{t}'(x) \sin(x) dx \geq  \tan(r) \cdot \varphi_{t}'(r),\\
-\varphi_{t}''(r) &= \tan(r) \cdot \varphi_{t}'(r) - (1-t)\chi'_{t}(r) \cos(r) \geq \tan(r) \cdot \varphi_{t}'(r).
\end{align*}
Together these yield the lower bound
\begin{equation*}
  \frac{-\varphi_{t}''(r)}{\varphi_{t}(r)} \geq \tan(r) \cdot \frac{\varphi_{t}'(r)}{\varphi_{t}(r)} \geq 1 - \frac{(1-t)}{t} \cdot \frac{\int_0^r -\chi_{t}'(x) \sin(x) dx}{\sin(r)}.
\end{equation*}
Using \eqref{eqdelta2} and rearranging, equation \eqref{Wcurvature} follows. Finally, to see equation \eqref{Wcollapse} we bound
\begin{align*}
\varphi_{t}(r) &\leq \tan(r) \cdot \varphi_{t}'(r) + \int_0^r -\chi_{t}'(x) \sin(x) dx\\
&\leq t \sin(r) + \chi_{t}(r) \sin(r) + t \sin(r)\\
&\leq 3t \sin(r),
\end{align*}
for any $r > t$, where in the second inequality we have used \eqref{eqdelta2}, and the third inequality we used \eqref{eqdelta1}. This completes the proof of Lemma~\ref{lemsmooth}.
\end{proof}

\begin{proof}[Proof of Theorem~\ref{thmlink}]
For any $t \in (0,1]$, consider the doubly warped product metrics on $[0,\frac{\pi}{2}] \times \Sp^{p-1} \times \Sp^{q-1}$ given by
$$\hat{h}(t) := dr^2 + \cos^2(r) g_{p-1} + \varphi^2_{t}(r) g_{q-1},$$
where $\varphi_{t}$ are the functions from Lemma~\ref{lemsmooth}. The conditions satisfied by $\varphi_t$ in \eqref{Wsmooth} are precisely those necessary to ensure that $\hat{h}(t)$ is a smooth metric. The curvature operator $\Rm(\hat{h}(t))$ has four distinct eigenvalues:
$$  1, \quad \frac{-\varphi''_{t}}{\varphi_t}, \quad \frac{1 - (\varphi'_t)^2}{\varphi_{t}^2}, \quad \tan \cdot \frac{\varphi_{t}'}{\varphi_{t}}.$$
Using the bounds $t \cos(r) \leq \varphi_{t}'(r) \leq \cos(r)$, $t \sin(r) \leq \varphi_t(r) \leq \sin(r)$, and the inequalities in \eqref{Wcurvature}, we conclude that $\Rm(\hat{h}(t)) \geq \max\{1-t,t\}$. In particular, the rescaled metrics $h(t) := \left( \sqrt{ \max\{1-t,t\}} \right)\cdot \hat{h}(t)$ satisfy $\Rm(h(t)) \geq 1$. The convergence as $t \downto 0$ is a simple consequence of \eqref{Wcollapse}.
\end{proof}

\section{Constructing the Steady solitons}\label{s4}
Recall that any metric cone whose link is a sphere $(\Sp^{n-1},h)$ with $\Rm(h) \geq 1$ can be uniquely resolved by a non-negatively curved expander modulo rescaling. The following lemma, originally due to Lai in the case of $O(1) \times O(n-1)$ symmetry \cite[Lemma 2.3]{Lai24}, shows that along a sequence of expanders whose links have collapsing volume, we can extract a steady soliton.

\begin{lem}\label{lemsteady}
Let $X_m:= (\Sp^{n-1}, h_m)$ denote a sequence of smooth metrics, such that
\begin{enumerate}
\item The curvature operator $\Rm(h_m) \geq 1$;
\item $\Vol (X_m) \downto 0$ as $m \rightarrow \infty$.\label{eq:voldecay}
\end{enumerate}
Denote by $(M_m, g_m , f_m, o_m)$ the expanding gradient Ricci soliton resolving the metric cone $C(X_m)$. After potentially passing to a subsequence, we have the smooth local convergence of the rescaled expanders
\begin{equation*}
    (M_m , \mathcal{R}_{g_m}(o_m) \cdot g_m , f_{m},o_m) \xrightarrow[C^\infty_{loc}]{m \rightarrow \infty} (M,g,f,o),
\end{equation*}
where $(M, g ,f)$ is a steady gradient Ricci soliton, and $o$ denotes a critical point of the potential function $f$.  Furthermore, if for some compact Lie group $G$, the links $X_m$ admit faithful isometric $G$-actions, then the steady soliton $(M,g,f)$ is $G$-symmetric.
\end{lem}

\begin{proof}
Since the rescaled expanders are complete with positive sectional curvature bounded above by one, by a result of Toponogov (see e.g. \cite{Top}) we have a uniform lower injectivity radius bound at each vertex $o_m$. Coupled with the global curvature bounds, by standard compactness theory a subsequence converges in a locally smooth sense to some complete Riemannian manifold
$$(M_m , \mathcal{R}_{g_m}(o_m) \cdot g_m ,o_m) \xrightarrow[C^\infty_{loc}]{m \rightarrow \infty} (M,g,o).$$
\begin{claim}
$\mathcal{R}_{g_m}(o_m) \upto \infty$ as $m \rightarrow \infty$.
\end{claim}
\begin{proof}[Proof of Claim]
If not, then after passing to a subsequence we have an upperbound $\mathcal{R}_{g_m}^{\half}(o_m) \leq C < \infty$. Coupled with Bishop-Gromov, the smooth local convergence and \eqref{eqsnn}, we find
\begin{align*}
\frac{\textnormal{vol}(B(o, C))}{C^n} = \lim_{m \upto \infty} \frac{\textnormal{vol}(B(o_m, C))}{C^n} &\leq \limsup_{m \upto \infty} \frac{\textnormal{vol}(B(o_m, \sqrt{\mathcal{R}_{g_m}(o_m)}))}{(\sqrt{\mathcal{R}_{g_m}(o_m)})^n}\\
&\leq \limsup_{m \upto \infty} 4^n \cdot \textnormal{AVR}(g_m),
\end{align*}
which contradicts our assumption that the asymptotic volume ratio of the expanders decays to zero.
\end{proof}
Since the rescaled metrics $\tilde{g}_m:= \mathcal{R}_{g_m}(o_m) \cdot g_m$ solve the equation
\begin{equation}\label{eq expseq}
 \Ric(\tilde{g}_m) + \frac{1}{2 \mathcal{R}_{g_m}(o_m)} \tilde{g}_m = \nabla^2 f_m,
\end{equation}
the local smooth convergence implies that the potential functions $f_m$ converge locally smoothly to a function $f$ on $M$, and hence $(M,g,f)$ satisfies the steady gradient Ricci soliton equation. It also follows trivially from the smooth convergence that $\nabla f(o) = 0$. Furthermore, if the links admit some faithful isometric $G$-action, then by the uniqueness in Theorem \ref{Deruelle thm}, the expanders resolving them must be $G$-symmetric. The symmetry of the limiting steady soliton then follows immediately from the locally smooth convergence.
\end{proof}

\begin{rmk}
In the previous lemma one may weaken the first assumption to a weak PIC1 curvature bound on the corresponding metric cones. In this case, we may instead utilise the existence theory from \cite{chan2024expanding} in order to extract a steady gradient Ricci soliton in the limit. However, due to the lack of a uniqueness result for such expanders, the preservation of symmetry in this case is currently unknown.
\end{rmk}

\begin{proof}[Proof of Theorem~\ref{thmpq}]
Let $h(t)$ for $t \in (0,1]$ be the family of $O(p) \times O(q)$ metrics defined on the sphere $\Sp^{p+q-1}$ constructed in Theorem~\ref{thmlink}. Fix $m \in \N$, and for each $t \in (0,1]$ define the constant $c_{m}(t) > 0$ so that
$$ \Vol(\Sp^{p+q-1}, c_m(t) \cdot h(t) ) = \Vol(\Sp^{p+q-1}, h(m^{-1})).$$
We then consider the new continuous family of $O(p) \times O(q)$ symmetric metrics
\begin{equation}
    h_{m}(t) := \min\{1, c_m(t)\} \cdot h(t), \quad \forall t \in (0,1].
\end{equation}
Note that the links $X_m(t) := (\Sp^{p+q-1},h_m(t))$ are a continuous family over $t \in (0,1]$ of smooth links with $O(p) \times O(q)$ symmetry and curvature operator bounded below by $1$, with $X_m(1)$ a scaled version of the round sphere. Therefore, by the existence and uniqueness theory of Deruelle, Theorem \ref{Deruelle thm}, we have a continuous family of positively curved $O(p) \times O(q)$ symmetric expanding gradient Ricci solitons resolving the metric cones $C(X_m(t))$, normalised to have unit scalar curvature at their vertices. Denote them by $E_{m}(t)$. For any sequence $t_m \in (0,1]^\N,$
\[
\Vol(X_m(t_m)) \leq \Vol(\Sp^{p+q-1}, h(m^{-1})) \downto 0,
\]
as $m \rightarrow \infty$, and so we can apply Lemma~\ref{lemsteady} to extract a $O(p) \times O(q)$ symmetric steady gradient soliton as the locally smooth limit along a subsequence of the expanders $E_m(t_m)$. In order to construct the soliton we want, we choose our sequence $(t_m)_{m \in \N}$ more carefully. We first consider two extremal cases.
\begin{itemize}
\item If $t_m \equiv 1$, because $X_m(1)$ and hence $E_m(1)$ are $O(n)$ symmetric, the sectional curvatures at the tip agree, and $\la_p = \la_q$.
\item If $t_m = \frac{1}{m}$, because $X_m(\frac{1}{m}) \rightarrow \Sp^{p}_+$, the limiting soliton must split $\R^{p}$. Moreover, since the sequence has unit scalar curvature at each vertex, the limiting soliton is not flat, and so the sectional curvature cannot be zero everywhere. It follows by the local smooth convergence at the vertex that for $m$ sufficiently large, $\la_p < \theta \cdot \la_q$ in $E_m(\frac{1}{m})$.
\end{itemize}
Since $t \mapsto E_m(t)$ is continuous for each $m$, we can find $t_m \in (\frac{1}{m},1)$ such that $\la_p = \theta \cdot \la_q$ in $E_m(t_m)$. It follows by the local smooth convergence that the steady gradient Ricci soliton extracted along any subsequence must also satisfy $\la_p = \th \cdot \la_q$ at the vertex. 
\end{proof}

\section{Geometry at infinity}\label{s5}
The aim of this section is to study the structure at infinity for the steady gradient Ricci solitons we constructed in Theorem~\ref{thmpq}. We start by proving the following.

\begin{thm}\label{ideal bdry thm}
For $p,q \in \Z_{\geq 2}$ and $\th \in (0,1)$, let $\mathcal{S}^{p,q}_\th$ be the steady gradient Ricci soliton constructed in Theorem~\ref{thmpq}. Then its ideal boundary has the form of a (not necessarily smooth) doubly warped product
$$\mathcal{S}^{p,q}_\th(\infty) = [0,L] \times_a \Sp^{p-1} \times_b \Sp^{q-1},$$
where $L \in [0,\pi/2]$ and $a,b : [0,L] \to [0,1]$ are $1$-Lipschitz concave functions, with $a$ non-increasing, $b$ non-decreasing, and $a(L) = b(0) = 0$.
\end{thm}

\begin{rmk}\label{rmk pdisk}
The quantity $\al := a(0)$ (resp. $\be:=b(L)$) corresponds to the length of the largest $O(p) \times \{1\}$ (resp. $\{1\} \times O(q)$) orbit inside the ideal boundary. If $\al > 0$, it would follow that $\be = 0$ since the asymptotic volume ratio of $\mathcal{S}^{p,q}_\th$ is zero (see, for instance, \cite{KleinerLott}), and hence 
$$ \mathcal{S}^{p,q}_\th(\infty) = [0,L] \times_a \Sp^{p-1} \cong \Sp^p_{+}.$$
\end{rmk}

By Lemma~\ref{lem Busemann} we have the Gromov-Hasdorff convergence of the level sets
$$ \left(\{B = s\},s^{-1}\cdot i_s\right)  \rightarrow \mathcal{S}^{p,q}_{\th}(\infty),$$
as $s \upto \infty$, where $B$ denotes the Busemann function about the vertex $o$. Moreover, it is clear that $B$ inherits the $O(p) \times O(q)$ symmetry of the soliton. Note however that although these hypersurfaces are convex, they may not be smooth. We therefore appeal to a mollification method of Alexander, Kapovitch and Petrunin \cite{AKP13} to approximate each of the level sets by smooth metrics on the sphere with non-negative sectional curvature.

\begin{lem}\label{lem symmetric approx}
There exists a sequence of smooth $O(p) \times O(q)$ symmetric metrics on the sphere with non-negative sectional curvature converging to the ideal boundary in the Gromov--Hausdorff sense.
\end{lem}

\begin{proof}
Applying \cite[][Theorem 1]{AKP13} directly gives a smooth approximation of the ideal boundary by metrics of non-negative sectional curvature. In order to ensure that the approximation retains the $O(p) \times O(q)$ symmetry of the ambient soliton, we appeal to Lemma~\ref{lem Gav} so that, for all $\la-$concave functions used in the proofs of \cite[][Lemma 4, Theorem 1]{AKP13}, we may replace them by their $O(p) \times O(q)-$average. The symmetry of the approximations is then immediate.
\end{proof}

\begin{proof}[Proof of Theorem \ref{ideal bdry thm}]
Let $h_k$ be a sequence of smooth $O(p) \times O(q)$ symmetric metrics on $\Sp^{p+q-1}$ with non-negative sectional curvature converging to the ideal boundary in the Gromov-Hausdorff sense, which exists by Lemma~\ref{lem symmetric approx}. By the standard theory of cohomogeneous one group actions, these metrics are doubly warped products
$$ (\Sp^{p+q-1},h_k) \cong  [0,L_k] \times_{a_k} \Sp^{p-1} \times_{b_k} \Sp^{q-1},$$
for some $L_k>0$ and smooth functions $a_k,b_k : [0,L_k] \to [0,\infty)$ satisfying certain constraints (see the discussion in \cite{BoehmEinstein,coh1Einstein} for the precise details). In particular, the functions $a_k,b_k$ are concave and $1$-Lipschitz, with $a_k$ decreasing and $b_k$ increasing, and with the boundary data $a_k(L_k) = b_k(0) = 0$.

It follows from Proposition \ref{prop; cbb1} that the ideal boundary is a connected Alexandrov space of curvature bounded below by one. By the metric space formulation of Bonnet-Myers \cite[Theorem 10.4.1]{BBI}, its diameter is bounded above by $\pi$. By the convergence to the ideal boundary, after passing to a subsequence we can ensure that $L_k \rightarrow L \in [0,\pi]$ as $k \rightarrow \infty$. Furthermore, by scaling the metrics $h_k$ if necessary, we may assume that $L_k$ is a decreasing sequence. It follows that the sequence of warping functions $a_k,b_k$ are uniformly bounded $1$-Lipschitz concave functions on $[0,L]$, and by Arzela-Ascoli, along a further subsequence we have uniform convergence to some limiting functions $a,b : [0,L] \to [0,L]$. Note that this convergence immediately implies the Gromov-Hausdorff convergence $$(\Sp^{p+q-1},h_k) \rightarrow [0,L] \times_a \Sp^{p-1} \times_b \Sp^{q-1},$$
as $k \rightarrow \infty$, which gives the ideal boundary its desired form. 

Finally, the leaf $([0,L] \times_{a} \Sp^{p-1}) \times \{y_0\}$ for a fixed $y_0 \in \Sp^{q-1}$ is a length space, and hence is an Alexandrov space of curvature bounded below by $1$ in its own right. Its boundary is the sphere $\{L\} \times \Sp^{p-1}$. Doubling along this boundary gives an Alexandrov space of curvature bounded below $1$, and hence applying the metric space formulation of Bonnet-Myers again, has diameter bounded above by $\pi$. In particular, we deduce that $2L \leq \pi$. Similarly, by applying the Bishop-Gromov theorem for Alexandrov spaces with curvature bounded below by $1$ \cite[][Theorem 10.6.6]{BBI}, it follows that $a,b \leq 1$.
\end{proof}

For the rest of this section, we fix $\mathcal{S}_{\th}^{p,q} = (M,g,f,o)$ for some $p,q \in \Z_{\geq 2}$ and $\th \in (0,1)$. Moreover, we assume that $\al > 0$ in  $\mathcal{S}_{\th}^{p,q}(\infty)$ as in Remark~\ref{rmk pdisk}. We will now perform a dimension reduction along a sequence of points $x_m \in \Sigma_{p}$ going to infinity rescaled by their scalar curvature. It follows by the work of Chan--Ma--Zhang that a non-collapsed steady gradient Ricci solitons with $\operatorname{sec}\geq 0$ always dimension reduces at infinity \cite{DimRedSteadySoltn}. However, we expect when $q=2$ that our soliton is collapsed. Instead of trying to follow \cite{DimRedSteadySoltn}, we shall utilise the symmetry to bypass the non-collapsing assumption as in \cite[Lemma 3.3]{Lai24}. Under our non-collapsing assumption on the ideal boundary, the limit will split more copies of $\R$. The proof is inspired by the ideas in \cite[Section 4]{Zhu23}. We begin with the following general result about Alexandrov spaces. 

\begin{lem}\label{lembu}
If $(X_m,d_m,x_m)$ is a sequence of Alexandrov spaces of non-negative curvature converging in the pointed Gromov--Hausdorff sense to $(X,d,x)$, and if $(T_xX,\rho,0)$ denotes the tangent cone of $(X,d)$ at $x$, then after passing to a subsequence, we have convergence of the unit balls
$$ \left( B(x_m,1) , m \cdot d_m, x_m \right) \rightarrow (B(0,1),\rho,0)$$
in the pointed Gromov--Hausdorff sense as $m\to \infty,$ where $B(0,1)\subset T_xX.$ 
\end{lem}

\begin{proof}
Using the definition of the tangent cone
$$  (B(x,1) , j\cdot d,x) \rightarrow (B(0,1),\rho,0),$$
in the pointed Gromov--Hausdorff sense as $j \upto \infty$. For each $j \in \N$, choose $m$ (depending on $j$) sufficiently large so that
$$ d_{GH}\left( (B(x_m,1), j \cdot d_m, x_m) , (B(x,1) , j\cdot d,x) \right) < \frac{1}{j}.$$
The result follows by applying a suitable diagonal argument. 
\end{proof}

For each $x \in M$, define the volume scale
$$ v(x) := \sup \{ r>0 : \Vol(B(x,r)) \geq \half \omega_n r^n \}.$$
To simplify notation, set $r_m :=  d_g(o,x_m)$, $v_m := v(x_m)$ and $R_{m} := \mathcal{R}_{g_m}(x_m)$.
\begin{claim}
Along any sequence $x_m$ tending to infinity, we have
\begin{equation}\label{eqnVD}
    \lim_{m \rightarrow \infty} \frac{v_m}{r_m} = 0.
\end{equation}
\end{claim}
\begin{proof}[Proof of Claim]
If we assume the claim fails, then after passing to a subsequence we may assume that $v_m \geq \eta \cdot r_m$ for some $\eta \in (0,1]$. It follows that
\begin{align*}
\textnormal{AVR}(g) &= \lim_{m \upto \infty} \frac{\textnormal{vol}(B_g(o,2r_m))}{ (2r_m)^n} \geq \limsup_{m \upto \infty} \frac{\textnormal{vol}(B_g(x_m,\eta\cdot r_m))}{ (2r_m)^n}\\
&\geq \left(\frac{\eta}{2}\right)^n \cdot \limsup_{m \upto \infty} \frac{\textnormal{vol}(B_g(x_m,\eta\cdot r_m))}{ (\eta \cdot r_m)^n}\\
&\geq \left(\frac{\eta}{2}\right)^n \cdot \limsup_{m \upto \infty} \frac{\textnormal{vol}(B_g(x_m,v_m))}{ (v_m)^n} = \left(\frac{\eta}{2}\right)^n \frac{\omega_n}{2} > 0,
\end{align*}
where in the fourth inequality we used Bishop-Gromov. Since the only ancient Ricci flows with non-negative curvature operator and positive asymptotic volume ratio are flat \cite{KleinerLott}, we have a contradiction.  
\end{proof}

We first consider the rescaled sequence $(M,v_m^{-2} \cdot g(v_m^2 t),x_m).$ By an estimate of Perelman \cite[Corollary 45.1(b)]{KleinerLott}, $\mathcal{R}_{v_m^{-2} \cdot g(v_m^2 t)}(x_m)$ is uniformly bounded above. Note that this bound persists backwards in time along the Ricci flow by the Harnack inequality \cite{HarnackHamilton}, and we therefore have a uniform injectivity radius lower bound along the sequence. By Hamilton's compactness theorem, we can extract an ancient Ricci flow $(M_\infty,g_\infty(t),x_\infty)$ along a subsequence.

\begin{prop}\label{prop vs}
After passing to a subsequence, the ancient Ricci flows $(M,v_m^{-2} \cdot g(v_m^2 t),x_m)$ smoothly converges to an ancient Ricci flow that splits $\R^p$. That is, the limit is of the form $(\R^p \times N^q, g_\infty(t),x_\infty)$, with $g_\infty(t) = g_0 \oplus g_N(t)$, where $g_0$ is the flat Euclidean metric on $\R^p$. Moreover, $(N^q,g_N(t))$ is a rotationally symmetric ancient Ricci flow with positive curvature operator.
\end{prop}

\begin{proof}
Using equation \eqref{eqnVD} we may assume that after passing to a subsequence 
$$ r_m \geq m^2 \cdot v_m.$$
We first note that a subsequence of $(M,r_m^{-2}g,o)$ converges in the pointed Gromov--Hausdorff sense to the metric cone $(C,d_C,o_C)$ over the ideal boundary. Moreover, after passing to possibly a further subsequence, we can ensure that $\tilde{x}_m$, the image of $x_m$ inside the cone $C$ under the Gromov-Hausdorff approximations, converges to a point $\tilde{x}_m \rightarrow \tilde{x}$, with $d_C(\tilde x,o_C) = 1$. In particular, we have the pointed Gromov--Hausdorff convergence of $(M,r_m^{-2}g,x_m)$ to $(C,d_C,\tilde{x})$. We can now apply Lemma~\ref{lembu} so that, again after possibly passing to a further subsequence, we have the pointed Gromov--Hausdorff convergence of the balls 
\begin{equation}\label{eq stc}
\left( B(x_m,\frac{r_m}{mv_m}), v_m^{-1} \cdot d_g, x_m\right) = (B(x_m,1), m r_m^{-1}d_g, x_m) \rightarrow (B(0,1),\rho,0),
\end{equation} 
where $(B(0,1),\rho)$ is the unit ball inside $(T_{\tilde{x}} C,\rho,0)$, the tangent cone of $C$ at $\tilde{x}$. Finally, by our assumption that $\al > 0$ in the ideal boundary, $T_{\tilde{x}}C$ splits $\R^p$. In particular, there exists points $y_1^{\pm},\ldots,y_p^{\pm} \in B(0,1)$ such that
\begin{equation*}
\begin{cases}
\rho(0,y_j^\pm) = \half &: \forall j \in \{1,\ldots,p\},\\
\rho(y_j^+,y_j^-) = 1 &: \forall j \in \{1,\ldots,p\},\\
\rho(y_j^{\pm},y_k^{\pm}) = \frac{1}{\sqrt{2}} &: \forall k \neq j \in \{1,\ldots,p\}.\\
\end{cases}
\end{equation*}
By the convergence in \eqref{eq stc}, there exists sequences of points $y_{m,1}^{\pm},\ldots,y_{m,p}^{\pm} \in M$ such that as $m \rightarrow \infty$ we have
\begin{equation*}
\begin{cases}
m r_m^{-1} d_g(x_m,y_{m,j}^\pm) \rightarrow \half &: \forall j \in \{1,\ldots,p\},\\
m r_m^{-1} d_g(y_{m,j}^+,y_{m,j}^-) \rightarrow 1 &: \forall j \in \{1,\ldots,p\},\\
m r_m^{-1} d_g(y_{m,j}^{\pm},y_{m,k}^{\pm}) \rightarrow \frac{1}{\sqrt{2}} &: \forall k \neq j \in \{1,\ldots,p\},\\
\end{cases}
\end{equation*}
and by the angle comparison theorem it follows that
\begin{equation*}
\begin{cases}
\liminf_{m \rightarrow \infty} \measuredangle y_{m,j}^+ x_m y_{m,j}^{-} \geq \pi &: \forall j \in \{1,\ldots,p\},\\
\liminf_{m \rightarrow \infty} \measuredangle y_{m,j}^{\pm} x_m y_{m,k}^{\pm} \geq \frac{\pi}{2} &: \forall k \neq j \in \{1,\ldots,p\}.\\
\end{cases}
\end{equation*}
In particular, arguing the same way as in \cite[Lemma 4.4]{Zhu23}, the geodesics $x_m y_{m,j}^{\pm}$ converge pointwise to some geodesic rays $\sigma_{j}^{\pm}:[0,\infty) \to (M_\infty,g_\infty(0))$ emanating from $\sigma_j^{\pm}(0) = x_\infty$, for $j \in \{1,\ldots,p\}$. Moreover, it follows from the previous inequalities that
\begin{equation*}
\begin{cases}
\lim_{s \rightarrow \infty} \measuredangle \sigma_{j}^+(s) x_\infty \sigma_{j}^{-}(s) = \pi &: \forall j \in \{1,\ldots,p\},\\
\liminf_{s \rightarrow \infty} \measuredangle \sigma_{j}^{\pm}(s) x_\infty \sigma_{k}^{\pm}(s) \geq \frac{\pi}{2} &: \forall k \neq j \in \{1,\ldots,p\}.\\
\end{cases}
\end{equation*}

The convergence in \eqref{eq stc} implies that $x_m$ is a $(p,\delta(m))$-strained point of size $\frac{m}{2}$, with $\delta(m) \downto 0$ as $m \rightarrow \infty$ (see for example \cite{BBI} for the precise definition of a strainer). That is, the limiting space splits $\R^p$. By the strong maximum principle, the Ricci flow splits $\R^p$ at all times. The symmetry and positivity of the curvature of $g_N(t)$ follows trivially from the smooth convergence.
\end{proof}

We have shown the appropriate splitting along our sequence of points going to infinity, up to rescaling by the volume scale. By using the symmetry, we claim that the volume scale and the curvature scale are in fact comparable.
\begin{claim}
There exists $c>0$ such that $R_m \cdot v_m^{2} \geq c$, for every $m \in \N$.
\end{claim}
\begin{proof}[Proof of Claim]
If not, then the limiting Ricci flow from Proposition~\ref{prop vs} satisfies $\mathcal{R}_{g_\infty}(x_\infty) = 0$, which implies $g_\infty(t)$ is flat by the strong maximum principle. However, since $g_N(t)$ is $O(q)$ symmetric, $(\R^p \times N^q , g_\infty(t))$ must isometric to flat Euclidean space $\R^{p+q}$ and we have a contradiction to $\Vol_{g_\infty}(B(x_\infty,1)) = \frac{\omega_n}{2}$.   
\end{proof}

Therefore, the curvature is comparable to our volume scale along the sequence $x_m$ and hence we may conclude that the sequence $(M, \mathcal{R}_{g}(x_m) g(t \mathcal{R}_g(x_m)^{-1}) , x_m)$
also splits $\R^p$, which finishes the proof of Theorem \ref{thmsplitting}.

\printbibliography

\end{document}